\documentclass[a4paper,10pt]{amsart}
\usepackage[english]{babel}
\usepackage{amsmath,amsthm,amssymb,amsfonts}
\usepackage{multicol}
\usepackage[titletoc]{appendix}

\usepackage[pdftex]{color}
\usepackage[bookmarks=true,hyperindex,pdftex,colorlinks, citecolor=blue,linkcolor=blue, urlcolor=blue]{hyperref}
\usepackage[dvipsnames]{xcolor}
\usepackage{mathtools}
\usepackage{graphicx}

\usepackage[normalem]{ulem}

\usepackage{tikz}
\usetikzlibrary{arrows,automata}
\usetikzlibrary{positioning}

\tikzset{
	confetti/.style={
		rectangle,
		draw=black, very thick,
		minimum height=4em,
		minimum width=6em,
		inner sep=1pt,
		text centered,
	},
	fletxa/.style={
		draw=black, very thick, ->
	},
}

\theoremstyle{definition}
\newtheorem{theorem}{Theorem}[section]
\newtheorem{cor}[theorem]{Corollary}
\newtheorem{prop}[theorem]{Proposition}

\newtheorem{rem}[theorem]{Remark}

\newtheorem{definition}[theorem]{Definition}

\newcommand{\N}{\mathbb{N}}

\newcommand{\K}{\mathbb{K}}

\newcommand{\e}{\varepsilon}

\renewcommand{\leq}{\leqslant}
\renewcommand{\geq}{\geqslant}

\renewcommand{\geq}{\geqslant}

\begin{document}
	\title{On some local Bishop-Phelps-Bollob\'as properties}
	
\dedicatory{Dedicated to the memory of Professor Victor Lomonosov}
	
\author[Dantas]{Sheldon Dantas}
\address[Dantas]{Department of Mathematics, Faculty of Electrical Engineering, Czech Technical University in Prague, Technick\'a 2, 166 27 Prague 6, Czech Republic \newline
	\href{http://orcid.org/0000-0001-8117-3760}{ORCID: \texttt{0000-0001-8117-3760}  }}
\email{\texttt{gildashe@fel.cvut.cz}}

\author[Kim]{Sun Kwang Kim}
\address[Kim]{Department of Mathematics, Chungbuk National University, 1 Chungdae-ro, Seowon-Gu, Cheongju, Chungbuk 28644, Republic of Korea \newline
	\href{http://orcid.org/0000-0002-9402-2002}{ORCID: \texttt{0000-0002-9402-2002}  }}
\email{\texttt{skk@chungbuk.ac.kr}}

\author[Lee]{Han Ju Lee}
\address[Lee]{Department of Mathematics Education, Dongguk University - Seoul, 04620 (Seoul), Republic of Korea \newline
	\href{http://orcid.org/0000-0001-9523-2987}{ORCID: \texttt{0000-0001-9523-2987}  }
}
\email{\texttt{hanjulee@dongguk.edu}}

\author[Mazzitelli]{Martin Mazzitelli}
\address[Mazzitelli]{Universidad Nacional del Comahue, CONICET, Departamento de Matem\'atica, Facultad de Econom\'ia y Administraci\'on, Neuqu\'en, Argentina.}
\email{\texttt{martin.mazzitelli@crub.uncoma.edu.ar}}

	\begin{abstract} 
We continue a line of study initiated in \cite{D, DKLM} about some \emph{local versions} of Bishop-Phelps-Bollob\'as type properties for bounded linear operators. We introduce and focus our attention on two of these local properties, which we call  {\bf {\bf L}$_{p, o}$} and {\bf {\bf L}$_{o, p}$}, and we explore the relation between them and some geometric properties of the underlying spaces, such as spaces having strict convexity, local uniform rotundity, and property $\beta$ of Lindenstrauss. At the end of the paper, we present a diagram comparing all the existing Bishop-Phelps-Bollob\'as type properties with each other. Some open questions are left throughout the article.
	\end{abstract}

	\date{\today}
	
	\thanks{The second author is the corresponding author. The first author was supported by the project OPVVV CAAS CZ.02.1.01/0.0/0.0/16\_019/0000778. The second author was partially supported by Basic Science Research Program through the National Research Foundation of Korea(NRF) funded by the Ministry of Education, Science and Technology (NRF-2017R1C1B1002928). The third author was supported by the Research program through the National Research Foundation of Korea funded by the Ministry of Education, Science and Technology (NRF-2016R1D1A1B03934771). The fourth author was partially supported by CONICET PIP 11220130100329CO}
	
	\subjclass[2010]{Primary 46B04; Secondary  46B07, 46B20}
	\keywords{Banach space; norm attaining operators; Bishop-Phelps-Bollob\'{a}s property}

	\maketitle
	
	\thispagestyle{plain}

	\section{Introduction} \label{int}

One of the main results in the theory of norm attaining functions defined on Banach spaces was proved by Errett Bishop and Robert R. Phelps in \cite{BP1}. They showed that the set of all functionals which attain the maximum on a nonempty closed bounded convex subset $S$ of a real Banach space $X$ is norm dense in the dual space $X^*$. On the other hand, Victor Lomonosov presented in \cite{Lom} an example which shows that this statement cannot be extended to the complex case by constructing a closed bounded convex subset of some Banach space with no support points. Here, we are interested to study this result when $S$ is the closed unit ball, which simply says  that the set of all norm attaining functionals defined on a real or complex Banach space $X$ is dense in $X^*$ (see also \cite{BP}). We will refer this last statement as the Bishop-Phelps theorem.
Joram Lindenstrauss was the first mathematician who considered the vector valued case of the Bishop-Phelps theorem (see \cite{Lind}). He presented a counterexample which proves that this theorem is no longer valid for bounded linear operators in general. Nevertheless, he gave some necessary conditions to get a Bishop-Phelps type theorem for this class of functions. For instance, if the domain $X$ is a reflexive Banach space, then it is true that the set of all norm attaining operators from $X$ into any Banach space $Y$ is dense in the set of all operators from $X$ into $Y$. After Lindenstrauss,  a lot of attention has been paid on this topic. We refer to the survey paper \cite{acostasur} and the references therein for more information about denseness of norm attaining functions in various directions.

In \cite{Bol}, B\'ela Bollob\'as proved a stronger version of the Bishop-Phelps theorem, in such a way that  whenever a norm-one functional $x^*$ almost attains its norm at some norm-one point $x$, it is possible to find a new norm-one functional $y^*$ and a new norm-one point $y$ such that $y^*$ attains its norm at $y$, $y$ is close to $x$, and $y^*$ is close to $x^*$. Since the norm of a functional is defined as a supremum and we can always take some point such that a given functional almost attains its norm, Bollob\'as result says that in the Bishop-Phelps theorem one can control the distances between the involved points and functionals. This result is known nowadays as the Bishop-Phelps-Bollob\'as theorem.
Motivated by Lindenstrauss work, in 2008, Mar\'ia Acosta, Richard Aron, Domingo Garc\'ia, and Manuel Maestre initiated the study of  the Bishop-Phelps-Bollob\'as theorem in the vector-valued case (see \cite{AAGM}). They found conditions on Banach spaces $X$ and $Y$ in order to get a Bishop-Phelps-Bollob\'as type theorem for operators from $X$ into $Y$. For instance, they characterized those spaces $Y$ such that  the Bishop-Phelps-Bollob\'as theorem holds for operators from $\ell_1$ into $Y$.
After more than 10 years of \cite{AAGM}, there is a huge literature about this topic and we refer the reader to \cite{acos, AMS, ACKLM, CKMM, CC, KL} and the references therein for further information.
Many different variants of the Bishop-Phelps-Bollob\'as theorem were introduced during the last years. For some of them, we refer the recent papers \cite{D, DKL, DKKLM, DKKLM1}. Our aim is to study local versions of these properties, as in \cite{DKLM}. Before we explain exactly what this means, let us introduce some notation and necessary preliminaries.

We work on Banach spaces over the field $\K$, which can be the real or complex numbers. We denote by $S_X$, $B_X$, and $X^*$ the unit sphere, the unit ball, and the topological dual of $X$, respectively. The symbol $\mathcal{L}(X, Y)$ stands for the set of all bounded linear operators from $X$ into $Y$ and we say that $T \in \mathcal{L}(X, Y)$ attains its norm (or it is norm attaining) if there is $x_0 \in S_X$ such that $$\|T\| = \sup_{x \in S_X}\|T(x)\| = \|T(x_0)\|.$$ Following \cite{AAGM}, we say that a pair of Banach spaces $(X, Y)$ satisfies the Bishop-Phelps-Bollob\'as property ({\bf BPBp}, for short) if given $\e > 0$, there is $\eta(\e) > 0$ such that whenever $T \in \mathcal{L}(X, Y)$ with $\|T\| = 1$ and $x \in S_X$ are such that $$\|T(x)\| > 1 - \eta(\e),$$ there are $S \in \mathcal{L}(X, Y)$ with $\|S\| = 1$ and $x_0 \in S_X$ such that 
$$
\|S(x_0)\| = 1,\quad \|x_0 - x\| < \e, \quad \text{and} \quad \|S - T\| < \e. 
$$
When $x_0 = x$ in the previous definition, we say that $(X, Y)$ has the Bishop-Phelps-Bollob\'as point property ({\bf BPBpp}, for short); this property was defined and studied in \cite{DKL, DKKLM}.
If instead of fixing the point $x$ (as in the {\bf BPBpp}) we fix the operator $T$, we say that $(X,Y)$ has the Bishop-Phelps-Bollob\'as operator property (see \cite{D, DKKLM1}). That is, $(X,Y)$ has the Bishop-Phelps-Bollob\'as operator property ({\bf BPBop}, for short) if given $\e > 0$, there is $\eta(\e) > 0$ such that whenever $T \in \mathcal{L}(X, Y)$ with $\|T\| = 1$ and $x_0 \in S_X$ are such that $\|T(x_0)\| > 1 - \eta(\e)$, there is $x_1 \in S_X$ such that  
$$
\|T(x_1)\| = 1\quad \text{and} \quad \|x_0 - x_1\| < \e. 
$$
Notice that the {\bf BPBp}, {\bf BPBpp}, and {\bf BPBop} are uniform properties in the sense that $\eta$ depends just on a given $\e > 0$. As we already mentioned before, we are interested to study the situations when $\eta$ depends not only $\e$, but also on the vector $x$ or the operator $T$. Some of them were already studied by the authors of the present paper in \cite{DKLM} and here we are using a similar notation.
We state now the definition of the two local properties on which we will focus. 
\begin{definition} \label{def} {\bf (a)} A pair $(X, Y)$ has the {\bf {\bf L}$_{p, o}$} if given $\e > 0$ and $T \in \mathcal{L}(X, Y)$ with $\|T\| = 1$, then there is $\eta(\e, T) > 0$ such that whenever $x \in S_X$ satisfies
\begin{equation*}
\|T (x)\| > 1 - \eta(\e, T),
\end{equation*}
there is $S \in \mathcal{L}(X, Y)$ with $\|S\| = 1$ such that
\begin{equation*}
\|S(x)\| = 1 \ \ \ \mbox{and} \ \ \ \|S - T\| < \e.
\end{equation*}
{\bf (b)} A pair $(X, Y)$ has the {\bf {\bf L}$_{o, p}$} if given $\e > 0$ and $x \in S_X$, then there is $\eta(\e, x) > 0$ such that whenever $T \in \mathcal{L}(X, Y)$ with $\|T\| = 1$ satisfies
\begin{equation*}
\|T(x)\| > 1 - \eta(\e, x),
\end{equation*}
there is $x_0 \in S_X$ such that
\begin{equation*}
\|T (x_0)\| = 1 \ \ \ \mbox{and} \ \ \ \|x_0 - x\| < \e.
\end{equation*}
\end{definition}
Let us clarify the notation: in the symbol {\bf {\bf L}$_{\square, \triangle}$}, both $\square$ and $\triangle$ can be $p$ or $o$, which are the initials of the words point and operator, respectively. If $(X, Y)$ has the {\bf {\bf L}$_{\square, \triangle}$}, then it means that we fix $\square$ and $\eta$ depends on $\triangle$. So, for instance, in {\bf {\bf L}$_{p, o}$} we fix a norm-one point $x$ and $\eta$ depends on a norm-one operator $T$. In \cite{DKLM}, properties {\bf {\bf L}$_{p, p}$} and {\bf {\bf L}$_{o, o}$} were addressed. Both of them are deeply related to geometric properties of the involved Banach spaces as local uniform rotundity or some of the Kadec-Klee properties. In fact, it turns out that the {\bf {\bf L}$_{p, p}$} for linear functionals defined on a Banach space $X$ is equivalent to the strong subdifferentiability of the norm of $X$ (see \cite[Theorem 2.3]{DKLM}). It is also a straightforward observation that if $X$ is reflexive then {\bf {\bf L}$_{o, o}$} is dual to {\bf {\bf L}$_{p, p}$} in the sense that $(X,\mathbb{K})$ has the {\bf {\bf L}$_{o, o}$} if and only if $(X^*,\mathbb{K})$ has the {\bf {\bf L}$_{p, p}$} (see \cite[Proposition~2.2]{DKLM}). 
Here, we are interested to give continuity in the study of these type of properties, but now investigating {\bf {\bf L}$_{p, o}$} and {\bf {\bf L}$_{o, p}$}.

We describe now the contents of this paper. In first place, we obtain sufficient and necessary conditions for a pair $(X,\mathbb{K})$ to have the {\bf {\bf L}$_{o, p}$}, in terms of some \emph{rotundity} properties of $X$. Specifically,  we prove that 
\begin{equation}\label{necess and suff}
\text{if $X$ is reflexive,}\quad\text{$X$ is LUR} \Rightarrow \text{ $(X, \K)$ has the {\bf L}$_{o, p}$} \Rightarrow \text{$X$ is strictly convex}.
\end{equation}
We also prove that there exists a dual relation between properties {\bf {\bf L}$_{p, o}$} and {\bf {\bf L}$_{o, p}$} in the functional case and, as a consequence, we get that if $X$ is reflexive and $X^*$ is locally uniformly rotund, then the pair $(X, \K)$ satisfies the {\bf {\bf L}$_{p, o}$}. As a consequence of \eqref{necess and suff} and the dual relation between {\bf {\bf L}$_{p, o}$} and {\bf {\bf L}$_{o, p}$} we see that, even for $2$-dimensional spaces, there is a Banach space $X$ such that the pair $(X, \K)$ fails both properties. This establish a difference between  the local properties {\bf {\bf L}$_{p, o}$}, {\bf {\bf L}$_{o, p}$} and {\bf {\bf L}$_{p, p}$}, {\bf {\bf L}$_{o, o}$}, since the latter properties hold in the finite-dimensional case.
Concerning linear operators, we show that pairs of the form $(X, \ell_{\infty}^2)$ and $(Z, Z)$, where $\dim(X) \geq 2$ and $Z$ is a $2$-dimensional space, fail property {\bf {\bf L}$_{o, p}$}. The situation with pairs like $(X, \ell_{\infty}^2)$ changes for property {\bf {\bf L}$_{p, o}$}: we prove that if $Y$ has property $\beta$ of Lindenstrauss with a finite index set $I$, then the pair $(X, Y)$ satisfies the {\bf {\bf L}$_{p, o}$} whenever $(X, \K)$ does. Nevertheless, this is no long true when $I$ is infinite and we present a counterexample to prove this. Finally, we show that $(\ell_1, Y)$ and $(c_0, Y)$ fail both properties for all Banach spaces $Y$. In the last part of the paper, we compare all these properties with each other and also with the {\bf BPBp}, {\bf BPBpp}, and {\bf BPBop}.

\section{The results}

In this section, we show the results we have for both properties {\bf {\bf L}$_{o, p}$} and {\bf {\bf L}$_{p, o}$}. We start by proving some positive results. Notice that it is clear that the {\bf BPBpp} implies the {\bf {\bf L}$_{p, o}$}. Hence, there are some immediate examples of pairs of Banach spaces $(X, Y)$ satisfying the {\bf {\bf L}$_{p, o}$} (see \cite{DKL, DKKLM} for positive results on the {\bf BPBpp}). It is also clear that the {\bf BPBop} implies the {\bf {\bf L}$_{o, p}$}, although this does not provide many examples, since the {\bf BPBop} holds only for the pairs $(\mathbb{K}, Y)$ for every Banach space $Y$ and $(X,\mathbb{K})$ for uniformly convex Banach spaces $X$ (see \cite{DKKLM, KL}).
Here, we get other examples of pairs $(X, Y)$ satisfying the properties {\bf {\bf L}$_{p, o}$} and {\bf {\bf L}$_{o, p}$} (see Proposition~\ref{17}, Corollary~\ref{18} and Theorems~\ref{betar} and \ref{comple}).
Recall that a Banach space is strictly convex if $\|\frac{x+y}{2}\| <1$ whenever $x,y \in S_X$, $x\neq y$, and that is 
locally uniformly rotund (LUR, for short) if for all $x, x_n \in S_X$, 
\begin{equation*}
\lim_n \|x_n + x\| = 2 \Longrightarrow \lim_n \|x_n - x\| = 0.
\end{equation*} 
It is a well-known fact that if $X$ is LUR, then is strictly convex. 

\begin{prop} \label{17} Let $X$ be a Banach space. 
\begin{enumerate}
\item[(i)] If $X$ is reflexive and LUR, then the pair $(X, \K)$ has the {\bf L}$_{o, p}$. 	
\item[(ii)] If $X$ has the Radon-Nikod\'ym property and $(X, \K)$ has the {\bf L}$_{o, p}$, then $X$ is strictly convex.
\end{enumerate}
\end{prop}

\begin{proof}
\noindent (i).
Otherwise, there are $\e_0 > 0$ and $x_0 \in S_X$ such that for every $n \in \N$, there is $x_n^* \in S_{X^*}$ with
\begin{equation*}
1 \geq \left|x_n^*(x_0)\right| \geq 1 - \frac{1}{n}	
\end{equation*}
such that whenever $x \in S_X$ satisfies $\|x - x_0\| < \e_0$, we have that $|x_n^*(x)| < 1$. Since $X$ is reflexive, there is $x_n \in S_X$ such that $|x_n^*(x_n)| = 1$ for every $n \in \N$. For suitable modulus 1 constants $c_n$, we have that
\begin{equation*}
1 \geq \left\| \frac{c_nx_n + x_0}{2} \right\| \geq \left|\frac{ x_n^*(c_nx_n) + x_n^*(x_0)}{2} \right| \longrightarrow 1.
\end{equation*}
Since $X$ is LUR, we see that $\|c_nx_n - x_0\| \longrightarrow 0$ as $n \rightarrow \infty$. Then, we must have $|x_n^*(c_nx_n)| < 1$ for large enough $n$ and this is a contradiction.

\noindent (ii). Let $\e>0$ and $x,y \in S_X$ such that $\|x-y\|\geq \e$. We want to show that there is $\delta(\varepsilon,x,y)>0$ such that $\frac{\|x+y\|}{2}\leq 1-\delta(\varepsilon,x,y)$. Let $\Gamma$ be the set of all bounded linear functionals in $S_{X^*}$ that strongly expose $B_{X^*}$. Following the proof of \cite [Theorem~2.1]{KL} (with slight modifications) we get that each $x^*\in \Gamma$ satisfies either
$$
\text{Re}\, x^*(x)\leq 1-\min\left\{\eta\left(\frac{\e^2}{64},x\right), \eta\left(\frac{\e^2}{64},y\right), \frac{\e^2}{64}\right\}
$$
or 
$$
\text{Re}\, x^*(y)\leq 1-\min\left\{\eta\left(\frac{\e^2}{64},x\right), \eta\left(\frac{\e^2}{64},y\right), \frac{\e^2}{64}\right\}.
$$
where $\eta(\cdot,\cdot)$ is the function in the definition of {\bf L}$_{o, p}$. 	
Now, since $X$ has the Radon-Nikod\'ym property we have that $\Gamma$ is dense in $S_{X^*}$ (see \cite{Bou, Ste}) and, consequently,
\begin{eqnarray*}
\frac{\|x+y\|}{2}&=&\sup\left\{ \text{Re}\, \frac{x^*(x)+x^*(y)}{2}:\,\, x^*\in \Gamma\right\}\\
&\leq& \frac{2-\min\left\{\eta\left(\frac{\e^2}{64},x\right), \eta\left(\frac{\e^2}{64},y\right), \frac{\e^2}{64}\right\}}{2}\\
&=&1-\frac{1}{2}\min\left\{\eta\left(\frac{\e^2}{64},x\right), \eta\left(\frac{\e^2}{64},y\right), \frac{\e^2}{64}\right\}.
\end{eqnarray*}
Then, $\delta(\varepsilon,x,y)=\frac{1}{2}\min\left\{\eta\left(\frac{\e^2}{64},x\right), \eta\left(\frac{\e^2}{64},y\right), \frac{\e^2}{64}\right\}$.
\end{proof}

We do not know if reflexivity (or the Radon-Nikod\'ym property) is a necessary condition for the {\bf L}$_{o, p}$ in the above proposition. However, if we assume that $X$ is reflexive, we have the following consequence.

\begin{cor} \label{18} Let $X$ be a reflexive Banach space. 
\begin{enumerate}
\item[(i)] If $X^*$ is LUR, then the pair $(X, \K)$ has the {\bf L}$_{p, o}$.	
\item[(ii)] If $(X, \K)$ has the {\bf L}$_{p, o}$, then $X^*$ is strictly convex.
\end{enumerate}
\end{cor}

Notice that (ii) is just a consequence of Proposition \ref{17}.(ii). To see (i), we prove dual relations between the properties {\bf L}$_{p, o}$ and {\bf L}$_{o, p}$ for functionals. Corollary~\ref{18}.(i) will then follow as a combination of Propositions \ref{17}.(i) and \ref{if}.

\begin{prop}\label{if} Let $X$ be a Banach space. If $(X^*, \K)$ has the {\bf L}$_{o, p}$, then $(X, \K)$ has the {\bf L}$_{p, o}$.
\end{prop}
\begin{proof} Assume $\e > 0$ and $x^* \in X^*$ with $\|x^*\| = 1$ are given. By hypothesis, we can take the constant $\eta(\e, x^*) > 0$ for the {\bf L}$_{o, p}$ of the pair  $(X^*, \K)$. Let $x \in S_X$ be such that $|x^*(x)| > 1 - \eta(\e, x^*)$. Using the canonical inclusion $\hat{~}:X \longrightarrow X^{**}$, we have $|\hat{x}(x^*)| = |x^*(x)| > 1 - \eta(\e, x^*)$, and so there exists $x_1^* \in S_{X^*}$ such that $|\hat{x}(x_1^*)| = |x_1^*(x)| = 1$ and $\|x_1^* - x^*\| < \e$. This proves that $(X, \K)$ has the {\bf L}$_{p, o}$.
\end{proof}

\begin{prop} Let $X$ be a reflexive Banach space. The pair $(X, \K)$ has the {\bf L}$_{p, o}$ if and only if $(X^*, \K)$ has the {\bf L}$_{o, p}$.
\label{equivalence1}	
\end{prop}

\begin{proof} From Proposition \ref{if}, we need to prove just the `only if' part. Assume $\e > 0$ and $x^* \in S_{X^*}$ are given.  By hypothesis, there is the constant $\eta(\e, x^*) > 0$ for the {\bf L}$_{p,o}$ of the pair  $(X, \K)$. Let $x^{**} \in X^{**}$ with $\|x^{**}\| = 1$ be such that $|x^{**}(x^*)| > 1 - \eta(\e, x^*)$. Using the canonical inclusion $\hat{~}:X \longrightarrow X^{**}$ and the reflexivity of $X$, there exists $x\in X$ such that $\hat{x}=x^{**}$. Hence, we have  $|x^{**}(x^*)|=|x^*(x)| > 1 - \eta(\e, x^*)$, and so there exists $z \in S_{X}$ such that $|x^*(z)| = 1$ and $\|z - x\| < \e$. The bidual element $\hat{z}$ is the desired one for the {\bf L}$_{o, p}$ of the pair $(X^*, \K)$.
\end{proof}

At this point we would like to stress some open problems that we are not able to solve. The first one was mentioned above. The second one relies on the fact that, those spaces $X$ for which we can assure that $(X, \mathbb{K})$ has the {\bf L}$_{o, p}$ (respectively, {\bf L}$_{p, o}$), satisfy also that $(X, \mathbb{K})$ has the {\bf L}$_{o, o}$ (respectively, {\bf L}$_{p, p}$). Indeed, it was already observed (see the discussion above \cite[Theorem~2.5]{DKLM}) that if $X$ is reflexive and LUR, then $(X, \mathbb{K})$ has the {\bf L}$_{o, o}$.

\noindent \textbf{Question 1:} Does property {\bf L}$_{o, p}$ (or {\bf L}$_{p, o}$) of the pair $(X, \mathbb{K})$ imply reflexivity of $X$?
\smallskip

\noindent \textbf{Question 2:} Does property {\bf L}$_{o, p}$ (respectively, {\bf L}$_{p, o}$) imply property {\bf L}$_{o, o}$ (respectively, {\bf L}$_{p, p}$)?
\smallskip

It is known that, for finite-dimensional Banach spaces $X$ and $Y$, the pair $(X, Y)$ satisfies property {\bf L}$_{p, p}$ (\cite[Proposition 2.9]{DKLM}). Besides this, it was proved in \cite{D} that if $X$ is finite dimensional, then the pair $(X, Y)$ has the {\bf L}$_{o, o}$ for every Banach space $Y$. However, this is not the case for both properties {\bf L}$_{p, o}$ and {\bf L}$_{o, p}$ even for linear functionals defined on 2-dimensional spaces. This is an immediate consequence of Proposition~\ref{equivalence1} and item (ii) in Proposition~\ref{17}.
 In what follows, we denote by $\ell_p^n$ the $n$-dimensional space endowed with the $p$-norm and $(e_i)_{i=1}^{n}$ their canonical basis.

\begin{prop} \label{7} The pairs $(\ell_1^n, \K)$ and $(\ell_\infty^n, \K)$ fail both {\bf L}$_{p, o}$ and {\bf L}$_{o, p}$ for $n\geq 2$.
\end{prop}

Next result shows that all the pairs of the form $(X, X)$, for 2-dimensional Banach spaces $X$ fails the {\bf L}$_{o, p}$ for linear operators.

\begin{prop} \label{6} Let $X$ be a $2$-dimensional Banach space. Then, the pair $(X, X)$ fails the {\bf L}$_{o, p}$.	
\end{prop}

\begin{proof} Consider $\{(v_1,v^*_1), (v_2,v^*_2)\}$ the Auerbach basis of the space $X$. Then, for every $x \in X$, we have that $x = v_1^*(x)v_1 + v_2^*(x)v_2$. Let us suppose by contradiction that the pair $(X, X)$ satisfies the {\bf L}$_{o, p}$ with some function $\eta(\cdot,\cdot)$ and let $n \in \N$ be such that $\frac{1}{n} < \eta(\e_0, v_1)$ for a fixed positive number $\e_0 \in (0, 1)$. Define $T_n: X \longrightarrow X$ by
	\begin{equation*}
		T_n(x) := \left(1 - \frac{1}{n} \right) v_1^*(x) v_1 + v_2^*(x)v_2 \ \ \ \left( x \in X \right).
	\end{equation*}
We see that
	\begin{eqnarray*}
		\|T_n(x)\| = \left\| \left(1 - \frac{1}{n} \right) v_1^*(x) v_1 + v_2^*(x)v_2 \right\| 
		&\leq& \left(1 - \frac{1}{n} \right) \left\| (v_1^*(x)v_1 + v_2^*(x)v_2) \right\| +  \frac{1}{n}\|v_2^*(x)v_2\| \\
		&\leq& \left(1 - \frac{1}{n} \right)\|x\| + \frac{1}{n} \leq 1
	\end{eqnarray*}
for arbitrary $x\in B_X$. This implies that $\|T_n\| = 1 = \|T_n(v_2)\|$. Now, since
\begin{equation*} 
	\|T_n(v_1)\| = 1 - \frac{1}{n} > 1 - \eta(\e_0, v_1), 
\end{equation*}	
there is $x_0 \in S_X$ such that $\|T_n(x_0)\| = 1$ and $\|x_0 - v_1\| < \e_0$. On the other hand, we have that
	\begin{equation*}
		1 = \|T_n(x_0)\| \leq \left(1 - \frac{1}{n} \right) \|x_0\| + \frac{1}{n}  |v_2^*(x_0)| \leq 1
	\end{equation*}
	which implies $|v_2^*(x_0)| = 1$. This gives us a contradiction since $1 > \e_0 > \|x_0 - v_1\| > |v_2^*(x_0)|$.
\end{proof}

We get another negative result for the property {\bf L}$_{o, p}$ when the range space is $\ell_{\infty}^2$. 

\begin{prop} \label{12} Let $X$ be a Banach space with $\dim(X) \geq 2$. Then, $(X, \ell_{\infty}^2)$ fails the {\bf L}$_{o, p}$.
\end{prop}

\begin{proof} Let $x_1, x_2 \in S_X$ and $x_1^*, x_2^* \in S_{X^*}$ be such that $x_i^*(x_j) = \delta_{ij}$ for $i, j = 1, 2$ (we may choose such elements by taking the Hahn-Banach extension of functionals of the Auerbach basis on a 2-dimensional subspace of $X$). We assume that the pair $(X, \ell_{\infty}^2)$  has the {\bf L}$_{o, p}$ with some function $\eta(\cdot,\cdot)$ and consider $n \in \N$ such that $\frac{1}{n} < \eta(\e_0, x_1)$ for a fixed positive number $\e_0 \in (0, 1)$. Define $T_n: X \longrightarrow \ell_{\infty}^2$ by
	\begin{equation*}
		T_n (x) := \left( \left(1 - \frac{1}{n} \right) x_1^*(x), x_2^*(x) \right) \ \ \ (x \in X).
	\end{equation*}
	Then $\|T_n\| \leq 1$ and $\|T_n(x_2)\|_{\infty} = 1$, which implies $\|T_n\| = 1$. Since
	\begin{equation*}
		\|T_n(x_1)\|_{\infty} = 1 - \frac{1}{n} > 1 - \eta(\e_0, x_1),	
	\end{equation*}
	there is $z \in S_X$ such that $\|T_n(z)\|_{\infty} = 1$ and $\|z - x_1\| < \e_0$. So, since
	\begin{equation*}
		1 = \|T_n(z)\|_{\infty} = \max \left\{ \left(1 - \frac{1}{n} \right) |x_1^*(z)|, |x_2^*(z)| \right\},
	\end{equation*}
	we have that $|x_2^*(z)| = 1$. Nevertheless, we have that $1 > \e_0 > \|z - x_1\| \geq |x_2^*(z) - x_2^*(x_1)| = |x_2^*(z)|$, which gives a contradiction.
\end{proof}
Taking into account Propositions~\ref{6} and \ref{12} and Corollary~\ref{l1c0} below, we leave the following open question.

\noindent \textbf{Question 3:} Are there spaces $X, Y$ with dim$(X)$, dim$(Y)$ $\geq 2$ such that $(X, Y)$ satisfies property {\bf L}$_{o, p}$?
\smallskip

Although we have a negative result in Proposition \ref{12} for the {\bf L}$_{o, p}$, the situation with property {\bf L}$_{p,o}$ is quite different. Indeed, we will prove that when we assume that the pair $(X, \K)$ has the {\bf L}$_{p, o}$, so does the pair $(X, \ell_{\infty}^2 )$. In fact, we get a more general result for Banach spaces satisfying property $\beta$ of Lindenstrauss (see \cite{Lind}). We say that a Banach space $Y$ has property $\beta$ with a index set $I$ and a constant $0 \leq \rho < 1$ if there is a set $\{(y_i,y_i^*): i \in I\} \subset S_{Y}\times S_{Y^*}$ such that

\begin{itemize}
\item $y_i^*(y_i) = 1$ for all $i \in I$,
\item $|y_i^*(y_j)| \leq \rho < 1$ for all $i, j \in I$ with $i \not= j$, and
\item $\|y\| = \sup_{i \in I} |y_i^*(y)|$ for all $y \in Y$.
\end{itemize}
Examples of Banach spaces satisfying such a property are $c_0(I)$ and $\ell_{\infty}(I)$ by taking $\{(e_i, e_i^*): i \in I \}$, the canonical biorthogonal system of these spaces. For the next result, we notice that in Definition \ref{def}.(a) one can use $T \in B_{\mathcal{L}(X, Y)}$ instead of $T \in S_{\mathcal{L}(X, Y)}$ by a simple change of parameters.

\begin{theorem} \label{betar}Let $X, Y$ be Banach spaces. Suppose that $(X,\mathbb{K})$ satisfies the {\bf L}$_{p, o}$ and assume that $Y$ has property $\beta$ with a finite index set $I$ and constant $\rho$. Then, the pair $(X, Y)$ has the {\bf L}$_{p, o}$.
\end{theorem}

\begin{proof} The proof is similar to \cite[Proposition 2.4]{DKL}, but we give the details for sake of completeness. Let $I$ be a finite set and $\{(y_i,y_i^*): i \in I\} \subset S_{Y}\times S_{Y^*}$ be the set of property $\beta$. Consider $\eta (\cdot,\cdot)$, the function for the pair $(X, \K)$, which satisfies the {\bf L}$_{p, o}$. For each $\e > 0$ and $T\in S_{\mathcal{L}(X,Y)}$, we define 
\begin{equation*}	
	\psi(\e,T)=\min_{i\in I}\{\eta (\e,y_i^*\circ T )\} > 0.
\end{equation*}
Fixed $\e_0>0$ and $T_0\in S_{\mathcal{L}(X,Y)}$, we choose $0<\xi < \frac{\e_0}{4}$ such that
\begin{equation}
1 + \rho\left(\frac{\e_0}{4} + \xi \right) < \left(1 + \frac{\e_0}{4} \right)(1 - \xi).
\label{propbeta}
\end{equation}
Now, let $x_0 \in S_X$ be such that $\|T_0(x_0)\| > 1 - \psi (\xi,T_0)$. By the definition of property $\beta$ and the construction of $\psi$, there exists $k \in I$ such that 
$$y_{k}^*(T_0(x_0)) > 1 - \psi (\xi,T_0)\geq  1 - \eta (\xi,y_k\circ T_0) .$$ 
Hence, there exists a functional $x_1^* \in S_{X^*}$ such that $|x_1^*(x_0)| = 1$ and $\|x_1^* - y_{k}^*\circ T_0^*\| < \xi$. Now, we define $U: X \longrightarrow Y$ by
\begin{equation*}
U(x) := T_0(x) + \left[ \left( 1 + \frac{\e_0}{4} \right) x_1^*(x) - y_{k}^*\circ T_0^*(x) \right] y_{k} \ \ \ (x \in X).
\end{equation*}
We have that $\|U - T_0\| < \frac{\e_0}{4} + \xi < \frac{\e_0}{2}$. Moreover, for arbitrary $j\neq k$, we have that
\begin{equation*}
\|y_{j}^*\circ U\| \leq 1 + \rho \left( \frac{\e_0}{4} + \xi \right) < \left(1 + \frac{\e_0}{4} \right)(1 - \xi) < 1 + \frac{\e_0}{4} \ \ \ \mbox{and} \ \ \  \|y_{k}^*\circ U\|=1 + \frac{\e_0}{4}.
\end{equation*}
Then, $U$ attains its norm at $x_0$ and so the operator $V := U/\|U\|$ is the one we were looking for.
\end{proof}

The main difference between \cite[Proposition 2.4]{DKL} and Theorem \ref{betar} is the cardinality of the index set $I$. Indeed, in \cite[Proposition 2.4]{DKL}, we see that the set $I$ does not need to be finite, since if $X$ is uniformly smooth, then the pair $(X, \K)$ satisfies the {\bf BPBpp} and so does the {\bf L}$_{p, o}$, which is, in this case, uniform, in the sense that $\eta$ depends only on a given $\e > 0$. This gives that $\psi(\e,T)=\inf_{i\in I}\{\eta (\e,y_i^*\circ T )\}$, in the proof of Theorem \ref{betar}, is strictly bigger than $0$. Naturally, one may ask whether the same result holds for infinite index sets. It turns out that this is not the case. To see why this happens, we consider the Banach space $X= \left[\oplus_{i=2}^\infty \ell_i^2\right]_{\ell_2}$, the $\ell_2$ direct sum of 2-dimensional $\ell_i$-spaces. We have that $X^*$ is a reflexive LUR Banach space (see, for example, \cite[Theorem 1.1]{L}). Hence, the pair $\left(\left[\oplus_{i=2}^\infty \ell_i^2\right]_{\ell_2}, \mathbb{K}\right)$ satisfies property {\bf L}$_{p, o}$ by Corollary \ref{18}. Recall that $\ell_{\infty}$ satisfies property $\beta$ with $I = \N$ and $\rho = 0$. Our counterexample is described in the next proposition.

\begin{prop} The pair $\left( \left[\oplus_{i=2}^\infty \ell_i^2\right]_{\ell_2},\ell_\infty\right)$ does not satisfy the {\bf L}$_{p, o}$.
\end{prop}

\begin{proof} We denote by $E_i$ and $\tilde{E}_i$ the natural embeddings from $\ell_i^2$ to $X$ and $(\ell_i^2)^*$ to $X^*$. Also we denote by $P_i$ the natural projections from $\ell_\infty$ to the $i$th coordinate. For $f_i^* = (1,0) \in S_{(\ell_i^2)^*}$, we define $T\in S_{\mathcal{L}(X,\ell_\infty)}$ by $T(\cdot)=(\tilde{E}_if_i^* (\cdot))_i$. Note that for each $z_i=\left(\frac{1}{2^{1/i}},\frac{1}{2^{1/i}}\right)\in S_{\ell_i^2}$, the element $z_i^*=\left(\frac{1}{2^{1-\frac{1}{i}}},\frac{1}{2^{1-\frac{1}{i}}}\right)$ is the unique norm-one functional so that $z_i^*(z_i)=1$. This shows that $\tilde{E}_iz_i^*$ is the unique element in $S_{X^*}$ so that $\tilde{E}_iz_i^*(E_iz_i)=1$, and then if an operator $S\in S_{\mathcal{L}(X,\ell_\infty)}$ attains its norm at $E_iz_i$, then there exists $j_0 \in \mathbb{N}$ and a modulus 1 scalar $c$ so that $P_{j_0} S=c\tilde{E}_iz_i^*$. From the construction, we see that 
\begin{equation*}	
\|T(E_iz_i)\| \longrightarrow 1 \ \ \ \mbox{and} \ \ \  \|P_jT-c\tilde{E}_iz_i^*\|>\frac{1}{2^{1-\frac{1}{i}}}
\end{equation*}
 for any modulus 1 scalar $c$ and $j\in \mathbb{N}$. 
	This proves that $\left( \left[\oplus_{i=2}^\infty \ell_i^2\right]_{\ell_2}, \ell_\infty\right)$ cannot satisfy the {\bf L}$_{p, o}$
\end{proof}

Next we give some results on stability concerning properties {\bf L}$_{p, o}$ and {\bf L}$_{o, p}$. 
Recall that a subspace $Z$ of a Banach space $X$ is one-complemented if $Z$ is the range of a norm-one projection on $X$.

\begin{theorem}\label{comple}
Let $X, Y$ be Banach spaces, and let $Z$ be a one-complemented subspace $Z$ of $X$. 
\begin{itemize}
\item[(i)] If the pair $(X,Y)$ has the {\bf L}$_{p, o}$, then so does $(Z,Y)$.
\item[(ii)] If the pair $(X,Y)$ has the {\bf L}$_{o, p}$, then so does $(Z,Y)$.
\end{itemize}
\end{theorem}

\begin{proof} We denote by $E$ and $P$ the canonical embedding and projection between $Z$ and $X$, respectively.

(i). Let $\e > 0$ and $T\in S_{\mathcal{L}(Z,Y)}$ be given. Assume that $z\in S_Z$ satisfy $\|T(z)\|>1-\eta(\e,T\circ P)$, where $\eta(\cdot,\cdot)$ is the function for the pair $(X, Y)$ having the {\bf L}$_{p, o}$. Since $\|(T\circ P)(E(z))\|=\|T(z)\|$ and $\|T\circ P\|=\|T\|$, there exists $S\in S_{\mathcal{L}(X,Y)}$ such that $\|S(E(z))\|=1$ and $\|S-T\circ P\|<\e$. Since $\|S\circ E-T\|\leq \|S-T\circ P\|$, we finish the proof.

\indent
(ii). Let $\e>0$ and $z\in S_Z$ be given. Assume that $T\in S_{\mathcal{L}(Z,Y)}$ satisfy $\|T(z)\| > 1 - \eta(\e,E(z))$, where $\eta(\cdot,\cdot)$ is the function for the pair $(X, Y)$ having the {\bf L}$_{o, p}$. Since $\|(T\circ P) (E(z))\|=\|T(z)\|$ and $\|T\circ P\|=\|T\|$, there exists $x\in S_X$ such that $\|x-E(z)\|<\e$ and $\|T\circ P (x)\|=1$. Since $\|P(x)-z\|\leq \|x-E(z)\|$, we finish the proof.
\end{proof}

\begin{prop}\label{func} Let $X$ and $Y$ be Banach spaces.
\begin{itemize}	
\item[(i)] If the pair $(X, Y)$ has the {\bf L}$_{o, p}$ for some Banach space $Y$, then so does $(X, \K)$.
\item[(ii)] If the pair $(X, Y)$ has the {\bf L}$_{p, o}$ for some Banach space $Y$, then so does $(X, \K)$.	
\end{itemize}
\end{prop}

\begin{proof} (i). Let $\e>0$ and $x \in S_X$ be given. By hypothesis, there is $\eta(\e, x) > 0$ for the pair $(X, Y)$. Let $x^* \in X^*$ with $\|x^*\| = 1$ be such that $|x^*(x)| > 1 - \eta(\e, x)$. Define $T \in \mathcal{L}(X, Y)$ by $T(z) := x^*(z)y_0$ for $z \in X$ and for a fixed $y_0 \in S_Y$. Then, $\|T\| = \|x^*\| = 1$ and $\|T(x)\| = |x^*(x)| > 1 - \eta(\e, x)$. So, there is $x_0 \in S_X$ such that $\|T(x_0)\| = |x^*(x_0)| = 1$ and $\|x_0 - x\| < \e$. This proves that $(X, \K)$ has the {\bf L}$_{o, p}$.

(ii). Let $\e>0$ and $x^* \in X^*$ with $\|x^*\| = 1$ be given. Again, define $T(z) := x^*(z) y_0$ for $z \in X$ and for a fixed $y_0 \in S_Y$. Set $\eta(\e, x^*) := \eta(\e, T) > 0$. Let $x_0 \in S_X$ be such that $|x^*(x_0)| > 1 -\eta(\e, x^*)$. Then, $\|T(x_0)\| > 1 - \eta(\e, T)$. So, there is $S \in \mathcal{L}(X, Y)$ with $\|S\| = 1$ such that $\|S(x_0)\| = 1$ and $\|S - T\| < \e$. Let $y_0^* \in S_{Y^*}$ be such that $y_0^*(S(x_0)) = \|S(x_0)\| = 1$. Set $x_1^* := S^* y_0^* \in S_{X^*}$. Then, $|x_1^*(x_0)| = 1$ and $\|x_1^* - x^*\| < \e$. This means that the pair $(X, \K)$ has the {\bf L}$_{p, o}$.	
\end{proof} 

By Proposition~\ref{7}, we know that the pairs $(\ell_1^2,\mathbb{K})$ and $(\ell_\infty^2,\mathbb{K})$ fails both {\bf L}$_{o, p}$ and {\bf L}$_{p, o}$. So, as a consequence of Theorem~\ref{comple} and Proposition~\ref{func}, if $X$ has $\ell_1^2$ or $\ell_\infty^2$ as a one-complemented subspace, then the pair $(X,Y)$ cannot have neither {\bf L}$_{o, p}$ nor {\bf L}$_{p, o}$ for all Banach spaces $Y$. Hence, we have the following consequence.

\begin{cor} \label{l1c0} Let $Y$ be a Banach space. The pairs $(\ell_1, Y)$ and $(c_0, Y)$ fail both {\bf L}$_{o, p}$ and {\bf L}$_{p, o}$.
\end{cor}

We finish the paper by discussing some of the relations between the Bishop-Phelps-Bollob\'as properties we mentioned so far. There are two more of them we would like to consider that we did not discuss in the present article. They are the local versions of the {\bf BPBp}, which we denote by {\bf L}$_{\triangle}$, where $\triangle$ means that the $\eta$ depends on a fixed point $x$ or on a fixed operator $T$. A pair of Banach spaces $(X, Y)$ has the {\bf {\bf L}$_p$} if given $\e > 0$ and $x \in S_X$, then there is $\eta(\e, x) > 0$ such that whenever $T \in \mathcal{L}(X, Y)$ with $\|T\| = 1$ satisfies
$\|T(x)\| > 1 - \eta(\e, x)$, there are $S \in \mathcal{L}(X, Y)$ with $\|S\| = 1$ and $x_0 \in S_X$ such that
	\begin{equation} \label{bpb}
		\|S (x_0)\| = 1, \ \ \ \|x_0 - x\| < \e, \ \ \ \mbox{and} \ \ \ \|S - T\| < \e.
\end{equation}	
On the other hand, $(X, Y)$ has the {\bf {\bf L}$_o$} if given $\e > 0$ and $T \in S_{\mathcal{L}(X, Y)}$, there is $\eta(\e, T) > 0$ such that whenever $x \in S_X$ satisfies $\|T (x)\| > 1 - \eta(\e, T)$, there are $S \in \mathcal{L}(X, Y)$ with $\|S\| = 1$ and $x_0 \in S_X$ such that (\ref{bpb}) holds. For more information about these properties, we refer the reader to \cite[Section 3]{DKLM}. In the next remark we compare the properties we have considered.

\begin{rem} We have the following observations.
\begin{enumerate}
\item[(i)] All the implications below between the Bishop-Phelps-Bollob\'as properties hold.
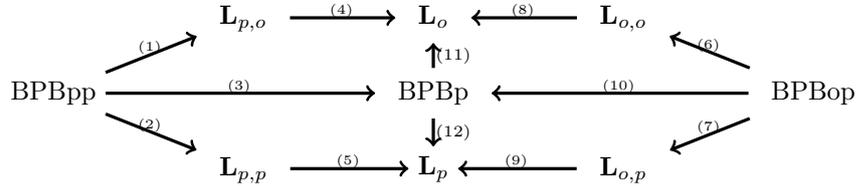
\begin{center} 	
	\begin{tikzpicture}[scale=0.5]
	\node (0) at (0,-2){BPBpp};
	\node (1) at (5,0)  {\begin{tabular}{c}
		{{\bf L}$_{p, o}$}
		\end{tabular}};
	\node (2) at (5,-4){\begin{tabular}{c}
		{{\bf L}$_{p,p}$}
		\end{tabular}};
	\node (3) at (10,0) {\begin{tabular}{c}
		{{\bf L}$_{o}$}
		\end{tabular}};
	\node (4) at (10,-4) {{\bf L}$_{p}$};
	\node (5) at (15,0) {\begin{tabular}{c}
		{{\bf L}$_{o, o}$}
		\end{tabular}};
	\node (6) at (15,-4) {\begin{tabular}{c}
		{{\bf L}$_{o,p}$}
		\end{tabular}};
	\node (7) at (20,-2) {\begin{tabular}{c}
		{BPBop}
		\end{tabular}};
	\node (8) at (10,-2) {\begin{tabular}{c}
		{BPBp}
		\end{tabular}};	
	
	\draw [fletxa] (0) to node [] {$^{^{(1)}}$}  (1);
	\draw [fletxa] (0) to node [] {$^{^{(2)}}$}  (2);
	\draw [fletxa] (1) to node [] {$^{^{(4)}}$}  (3);
	\draw [fletxa] (2) to node [] {$^{^{(5)}}$}  (4);
	\draw [fletxa] (5) to node [] {$^{^{(8)}}$}  (3);
	\draw [fletxa] (6) to node [] {$^{^{(9)}}$}  (4);
	\draw [fletxa] (7)  to node [] {$^{^{(6)}}$}  (5);
	\draw [fletxa] (7)  to node [] {$^{^{(7)}}$}  (6);
	\draw [fletxa] (0) to node [] {$^{^{(3)}}$}  (8);
	\draw [fletxa] (7) to node [] {$^{^{(10)}}$}  (8);
	\draw [fletxa] (8)  to node [] {\,\,\,\,\,\,\, \tiny{(11)}}  (3);
	\draw [fletxa] (8)  to node [] {\,\,\,\,\,\,\, \tiny{(12)}}  (4);
	\end{tikzpicture}
\end{center}
On the other hand, the reverse implications are not true.

\item[(ii)] The {\bf BPBp} does not imply {\bf {\bf L}$_{\square, \triangle}$}, where $\square$ and $\triangle$ can be $p$ or $o$.

\item[(iii)] There is no relation between properties {\bf {\bf L}$_{o, p}$} and {\bf {\bf L}$_{p, o}$}.
\item[(iv)] The {\bf {\bf L}$_{p, p}$} does not imply the {\bf {\bf L}$_{p, o}$}, but we do not know whether the {\bf {\bf L}$_{p, o}$} implies (or not) the {\bf {\bf L}$_{p, p}$}.
\item[(v)] The {\bf {\bf L}$_{o, o}$} does not imply the {\bf {\bf L}$_{o, p}$}, but we do not know whether the {\bf {\bf L}$_{o, p}$} implies (or not) the {\bf {\bf L}$_{o, o}$}.
\end{enumerate}
\end{rem}

We briefly discuss the statements in the above remark. It is clear that all the implications in (i) are satisfied, so let us show that the reverse implications do not hold.  In \cite[Section~5]{DKLM} it is proved that the reverse implications of (2), (3), (5), (6), (8), (10), (11) and (12) do not hold. The reverse implication of (4)  (respectively (9)) fails since, for instance, the pairs $(\ell_1,\mathbb{K})$ or $(c_0,\mathbb{K})$ have the {\bf {\bf L}$_{o}$} (respectively {\bf {\bf L}$_{p}$}) but fail the {\bf {\bf L}$_{p, o}$} (respectively {\bf {\bf L}$_{o, p}$}). To show that the reverse implication of (7) fails, just take a pair $(X, \mathbb{K})$ with $X$ reflexive and LUR but not uniformly convex. Analogously (reasoning with $X^*$ instead of $X$) we see that the reverse implication of (1) does not hold. To see (ii), just note that $(X, \mathbb{K})$ has the {\bf BPBp} for every Banach space $X$, which is clearly not true for any of the properties {\bf {\bf L}$_{\square, \triangle}$}. For (iii), take $X$ a uniformly smooth Banach space with $\dim(X) \geq 2$. Then, we have that $(X, \ell_{\infty}^2)$ has the {\bf BPBpp} (see \cite[Proposition 2.4]{DKL}) and, consequently, the {\bf L}$_{p, o}$, but fails the {\bf L}$_{o, p}$ in virtue of Proposition~\ref{12}. This shows that the {\bf L}$_{p, o}$ does not imply the {\bf L}$_{o, p}$. For the converse, take any finite-dimensional space $X$ which is strictly convex but not smooth. Then, $X^*$ cannot be strictly convex and by Corollary \ref{18}.(ii), the pair $(X, \K)$ fails property {\bf L}$_{p, o}$. On the other hand, by using Proposition \ref{17}.(i), it satisfies property {\bf L}$_{o, p}$. So, the {\bf L}$_{o, p}$ cannot imply the {\bf L}$_{p, o}$. Finally, for (iv) and (v), notice that the {\bf {\bf L}$_{p, p}$} cannot imply the {\bf {\bf L}$_{p, o}$} since $(\ell_1^2, \K)$ has the {\bf {\bf L}$_{p, p}$} (see \cite[Proposition 2.9]{DKLM}) but fails the {\bf {\bf L}$_{p, o}$} (see Proposition \ref{7}). The same example shows that {\bf {\bf L}$_{o, o}$} does not imply the {\bf {\bf L}$_{o, p}$}.

\end{document}